\documentclass[3p,10pt]{elsarticle}

\usepackage{amssymb,amsthm,latexsym,amsmath,subfigure,color,siunitx}
\usepackage{lineno}

\journal{}

\newcommand{\eps}{\varepsilon}
\newcommand{\epsb}{\varepsilon_{\mathrm{b}}}
\newcommand{\mub}{\mu_{\mathrm{b}}}
\newcommand{\sigmab}{\sigma_{\mathrm{b}}}
\newcommand{\kb}{k_{\mathrm{b}}}
\newcommand{\set}[1]{\left\{#1\right\}}

\newcommand{\p}{\partial}
\newcommand{\qand}{\quad\text{and}\quad}
\newcommand{\me}{\mathbf{e}}
\newcommand{\mr}{\mathbf{r}}
\newcommand{\mx}{\mathbf{x}}
\newcommand{\my}{\mathbf{y}}

\newcommand{\mGR}{\mathbf{G}_{\mathcal{R}}}
\newcommand{\mGE}{\mathbf{G}_{\mathcal{E}}}
\newcommand{\rd}{\mathrm{d}}
\newcommand{\vt}{\boldsymbol{\theta}}
\newcommand{\vv}{\boldsymbol{\vartheta}}

\DeclareMathOperator*{\inc}{inc}

\DeclareMathOperator*{\scat}{scat}

\theoremstyle{plain}
\newtheorem{theorem}{Theorem}[section]
\newtheorem{corollary}{Corollary}[section]

\theoremstyle{remark}

\begin{document}

\begin{frontmatter}



\title{Unveiling novel insights into Kirchhoff migration for effective object detection using experimental Fresnel dataset}
%On the imaging of short sound-soft open arc via orthogonality sampling method

%\author{Sangwoo Kang}
%\ead{sangwoo.kang@kaist.ac.kr}
%\address{Agency for Defense Development, Daejeon, Korea}
\author{Won-Kwang Park}
\ead{parkwk@kookmin.ac.kr}
\address{Department of Information Security, Cryptography, and Mathematics, Kookmin University, Seoul, 02707, Korea}

\begin{abstract}
This study investigates the applicability of Kirchhoff migration (KM) for a fast identification of unknown objects in a real-world limited-aperture inverse scattering problem. To demonstrate the theoretical basis for the applicability including unique determination of objects, the imaging function of the KM was formulated using a uniformly convergent infinite series of Bessel functions of integer order of the first kind based on the integral equation formula for the scattered field. Numerical simulations performed using the experimental Fresnel dataset are exhibited to achieve the theoretical results.
\end{abstract}


\end{frontmatter}



%% The Appendix part is started with the command \appendix;
%% appendix sections are then performed as normal sections


%\linenumbers

\section{Introduction}\label{sec:1}
The development of effective and stable imaging techniques from measurement data in inverse problems has been an ongoing challenge for decades because it is remarkably related to human life and involves several interesting research topics in the fields of mathematics, engineering, physics, and geophysics \cite{A1, CK, N3, Z2}. Similarly, many authors have developed theoretical algorithms, but most of them suffered from ``inverse crimes" because the numerical simulations were performed using artificial synthetic data to support theoretical results. Therefore, various research teams have provided databases (e.g., Ipswich data \cite{MK3} and Fresnel data \cite{BS, GSE}) obtained from laboratory-controlled experiments and tested algorithms on real-world experimental data, which are of great interest to the inverse scattering research community.

Kirchhoff migration (KM) is a well-known and promising non-iterative technique for determining the existence and outline shape of a set of arbitrary objects in inverse scattering problems and microwave imaging. From the pioneering work by Schneider \cite{S6}, the KM has been applied various problems such as imaging of complex structures \cite{B9}, terahertz imaging \cite{DJRBSM}, imaging of the volcanic field using teleseismic data \cite{AHFSBRTL}, radar imaging \cite{ZYSL}, imaging of cracks with planar array antennas \cite{AGKPS}, and imaging of scatterers using phaseless data \cite{BV}. Although it has been confirmed that KM is a fast, stable, and effective algorithm, complete elements of the so-called multi-static response (MSR) matrix must be available. Unfortunately, many elements, including the diagonal of the MSR matrix, are unusable in experimental Fresnel dataset thus, to the best of our knowledge, the theoretical basis for applicability has not been theoretically investigated yet. Nevertheless, based on the simulation results with Fresnel experimental data, the existence and outline shape of small objects can be recognized using KM.

In this study, we consider the use of KM to identify a set of small objects from Fresnel experimental dataset. To this end, we introduce the imaging function of the KM and demonstrate that it is composed of the square of the Bessel function of order zero and the infinite series of the Bessel function of nonzero order. Following the investigation of the imaging function's structure, we provide theoretical answers to the applicability of the KM and the unique determination of the object at a specific frequency.

This study is organized as follows. Section \ref{sec:2} discusses the two-dimensional direct scattering problem with a set of small objects and introduces the imaging function of the KM from the generated MSR matrix. In Section \ref{sec:3}, we demonstrate that the imaging function is composed of the square of the Bessel function of order zero and the infinite series of the Bessel function of nonzero order. Section \ref{sec:4} presents simulation results at various frequencies using experimental dataset to support the theoretical result.

\section{Direct scattering problem and imaging function of the Kirchhoff migation}\label{sec:2}
Let $D_s$, $s=1,2,\cdots,S$, be a small object completely embedded in a homogeneous domain $\Omega\subset\mathbb{R}^2$, and $D$ represent the collection of $D_s$. Throughout this study, we assume that the objects are well separated and that $D_s$ and $\Omega$ are characterized by their dielectric permittivity values at a given angular frequency $\omega=2\pi f$, where $f$ is the ordinary frequency. Based on the simulation setup of \cite{BS}, the conductivity, permeability, and permittivity values of $\Omega$ are set to $\sigmab=0$, $\mub=4\pi\times\SI{e-7}{\henry/\meter}$, and $\epsb=\SI{8.854e-12}{\farad/\meter}$, respectively, and we denote $\eps_s$ as the permittivity of $D_s$. Given the preceding, we present the following piecewise constant:
\[\eps(\mx)=\left\{\begin{array}{ccl}
\eps_s&\text{for}&\mx\in D_s\\
\epsb&\text{for}&\mx\in\Omega\backslash\overline{D}
\end{array}\right.\]
and let $\kb=\omega\sqrt{\epsb\mub}$ be the background wavenumber.

According to \cite{BS}, the range of receivers $\mathcal{R}_n$, $n=1,2,\cdots,N(=72)$ is restricted from $\SI{60}{\degree}$ to $\SI{300}{\degree}$, with a step size of $\SI{5}{\degree}$ based on each direction of the emitters $\mathcal{E}_m$, $m=1,2,\cdots,M(=36)$. The emitters are evenly distributed with step sizes of $\SI{10}{\degree}$ between $\SI{0}{\degree}$ and $\SI{350}{\degree}$. Thus, the positions of $\mathcal{E}_m$ and $\mathcal{R}_n$ are represented as follows:
\begin{align*}
\me_m&=E\vv_m=\SI{0.76}{\meter}(\cos\vartheta_m,\sin\vartheta_m),\quad\vartheta_m=\SI{10}{\degree}(m-1),\quad m=1,2,\cdots,M\\
\mr_n&=R\vt_n=\SI{0.72}{\meter}\big(\cos(\vartheta_m+\theta_n),\sin(\vartheta_m+\theta_n)\big),\quad\theta_n=\SI{5}{\degree}(n-1),\quad n=1,2,\cdots,N.
\end{align*}

Let $u_{\inc}(\mx,\me_m)$ be the incident field due to the point source at $\mathcal{E}_m$ and $u(\mr_n,\mx)$ be the corresponding total time-harmonic field that satisfies
\[\triangle u(\mr_n,\mx)+\omega^2\mub\eps(\mx)u(\mr_n,\mx)=0\quad\text{in}\quad\Omega\]
with the transmission condition for $\p D_s$. Here, the time-harmonic $e^{-i\omega t}$ is assumed. Let $u_{\scat}(\mr_n,\me_m)$ denote the measured scattered field that satisfies the Sommerfeld radiation condition. Since every $D_s$ is a small object, the Born approximation can be used to represent $u_{\scat}(\mr_n,\me_m)$ as the following integral equation:
\begin{equation}\label{IntegralEquation}
u_{\scat}(\mr_n,\me_m)=\kb^2\int_D\left(\frac{\eps(\my)-\epsb}{\epsb\mub}\right)G(\mr_n,\my)G(\my,\me_m)\rd\my,\quad G(\mx,\my)=-\frac{i}{4}H_0^{(1)}(\kb|\mx-\my|).
\end{equation}
Here, $H_0^{(1)}$ denotes the Hankel function of order zero of the first kind, refer to \cite{BCS}.

Let $\mathbb{K}$ be the generated MSR matrix, with elements $u_{\scat}(\mr_n,\me_m)$. Then, based on the \eqref{IntegralEquation}, we introduce the following test vectors: for each $\mx\in\Omega$,
\[\mGR(\mx)=\Big(\overline{G(\mx,\mr_1)},\overline{G(\mx,\mr_2)},\ldots,\overline{G(\mx,\mr_N)}\Big)
\qand
\mGE(\mx)=\Big(\overline{G(\mx,\me_1)},\overline{G(\mx,\me_2)},\ldots,\overline{G(\mx,\me_M)}\Big)^T.\]
The traditional imaging function of KM is defined as \[\mathfrak{F}(\mx)=|\mGR(\mx)\mathbb{K}\mGE(\mx)|,\quad\mr\in\Omega.\]
Then, the map of $\mathfrak{F}(\mx)$ will contain a peak of large magnitude if $\mx\in D$, allowing to recognize the existence or outline shape of objects (refer to \cite{AGKPS}). 

Unlike traditional studies, complete elements of $\mathbb{K}$ cannot be handled due to the limited range of receivers. In this study, when an emitter sends waves toward the target, only $49$ scattered field data are measurable instead of complete $72$ scattered field data. Thus, by converting the $23$ missing scattered field data to zero, the MSR matrix $\mathbb{K}$ takes the following form:
\[\mathbb{K}=(\phi_{n,m})\in\mathbb{C}^{N\times M},\quad\text{where}\quad \phi_{n,m}=\left\{\begin{array}{cl}
\smallskip u_{\scat}(\mr_n,\me_m),&n\in \mathcal{I}(m)\\
0,&\text{otherwise}.
\end{array}\right.\]
Here, $\mathcal{I}(m)=\set{2m+p~(\text{mod}~72):p=11,12,\ldots,59}$ represents an index set, refer to Figure \ref{MSR} for an illustration. Although the complete elements of $\mathbb{K}$ cannot be used, the objects can still be retrieved using the map of $\mathfrak{F}(\mx)$.

\begin{figure}[h]
\begin{center}
\includegraphics[width=.25\textwidth]{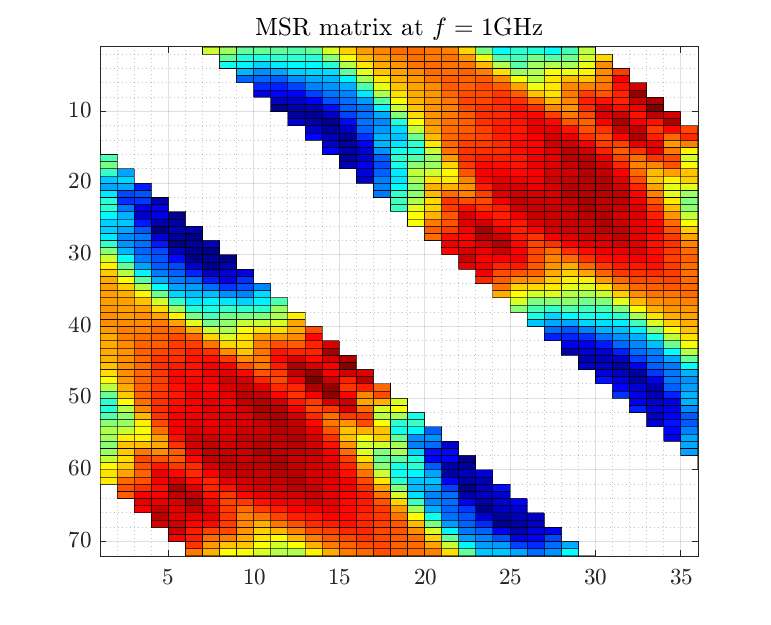}\hfill
\includegraphics[width=.25\textwidth]{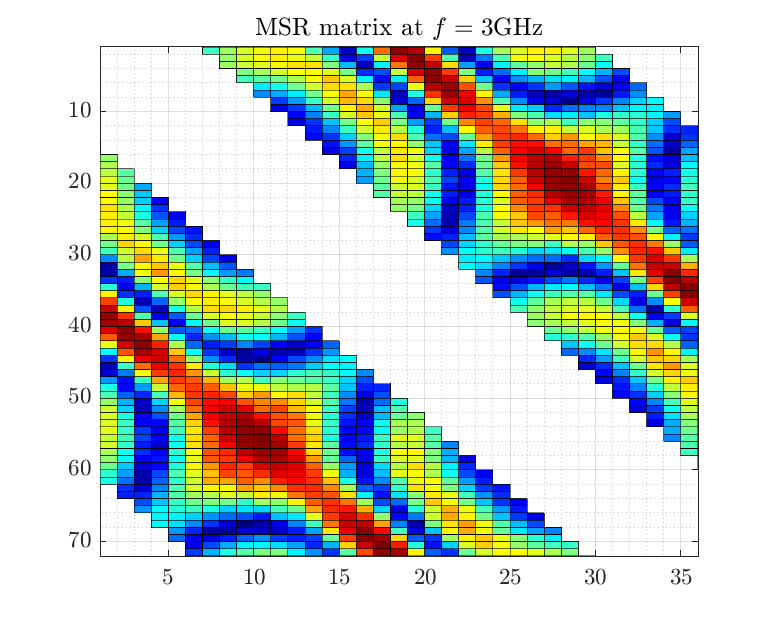}\hfill
\includegraphics[width=.25\textwidth]{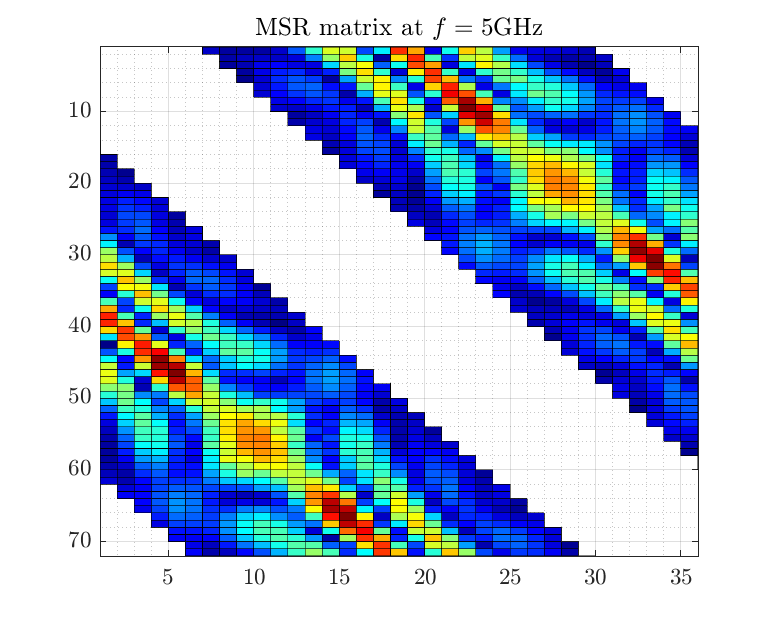}\hfill
\includegraphics[width=.25\textwidth]{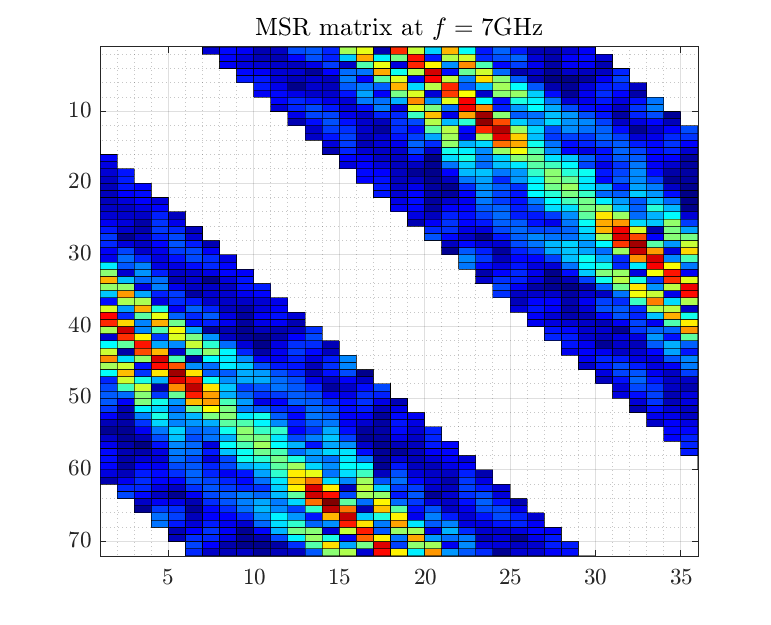}
\caption{\label{MSR}Illustration of generated MSR matrices.}
\end{center}
\end{figure}

%%
%$\mathcal{E}_m$: $m$th emitter and $\mathcal{R}_n$: $n$th receiver. Define an index set
%\[\]
%for example,
%\begin{align*}
%I(1)&=\set{13,14,15,16,\ldots,58,59,60,61}\\
%I(2)&=\set{15,16,17,18,\ldots,60,61,62,63}\\
%I(3)&=\set{17,18,19,20,\ldots,62,63,64,65}\\
%&\vdots\\
%I(7)&=\set{1,25,26,27,28,\ldots,69,70,71,72}\\
%I(8)&=\set{1,2,3,27,28,\ldots,69,70,71,72}.
%\end{align*}

\section{Applicability of the Kirchhoff migration with experimental dataset}\label{sec:3}

\begin{theorem}\label{Theorem}
Let $\mx-\my=|\mx-\my|(\cos\phi,\sin\phi)$ and assume that $4\kb|\mx-\mr_n|\gg1$ for all $n$ and $\mx\in\Omega$. *Therefore, $\mathfrak{F}(\mx)$ can be expressed as follows:
  \begin{equation}\label{Structure}
    \mathfrak{F}(\mx)\approx\left|\frac{1}{24RE}\int_D\left(\frac{\eps(\mr)-\epsb}{\epsb\mub}\right)\left[J_0(\kb|\mx-\my|)^2+\frac{3}{\pi}\sum_{p=1}^{\infty}\frac{1}{p}J_p(\kb|\mx-\my|)^2\sin\left(\frac{2p}{3}\pi\right)\right]\rd\my\right|,
  \end{equation}
where $J_p$ denotes the Bessel function of order $p$.
%\begin{equation}\label{DisturbFactor}
%\Lambda(\mx,\my)=\sum_{p=1}^{\infty}\frac{1}{p}J_p(\kb|\mx-\my|)^2\sin\left(\frac{2p}{3}\pi\right).
%\end{equation}
\end{theorem}
\begin{proof}
Since $4\kb|\my-\mr_n|\gg1$ for $n=1,2,\cdots,N$, the following asymptotic form holds (see \cite{CK} for example)
\begin{equation}\label{Asymptotic_Hankel}
H_0^{(1)}(\kb|\mx-\mr_n|)=\frac{(1-i)e^{i\kb R}}{\sqrt{\kb R\pi}}e^{-i\kb\vt_n\cdot\mx}\qand H_0^{(1)}(\kb|\mx-\me_m|)=\frac{(1-i)e^{i\kb E}}{\sqrt{\kb E\pi}}e^{-i\kb\vv_m\cdot\mx},
\end{equation}
we can examine that
\[\mGR(\mx)=\frac{(-1+i)e^{-i\kb R}}{4\sqrt{\kb R\pi}}
\Big(e^{i\kb\vt_1\cdot\mx},\ldots,
e^{i\kb\vt_N\cdot\mx}\Big),\quad
\mGE(\mx)=\frac{(-1+i)e^{-i\kb E}}{4\sqrt{\kb E\pi}}\Big(
e^{i\kb\vv_1\cdot\mx},\ldots
e^{i\kb\vv_M\cdot\mx}\Big)^T,\]
and
\[u_{\scat}(\mr_n,\me_m)=-\frac{i\kb e^{i\kb(R+E)}}{8\pi\sqrt{RE}}\int_D\left(\frac{\eps(\my)-\epsb}{\epsb\mub}\right)e^{-i\kb(\vt_n+\vv_m)\cdot\my}\rd\my.\]
Given the preceding, we can evaluate
\[\mGR(\mx)\mathbb{K}=\frac{(1+i)\kb e^{i\kb E}}{32\pi R\sqrt{\kb E\pi}}\begin{pmatrix}
\medskip\displaystyle\int_D\left(\frac{\eps(\mr)-\epsb}{\epsb\mub}\right)e^{-i
\kb\vv_1\cdot\my}\bigg[\sum_{n\in I(1)}e^{ik\vt_n\cdot(\mx-\my)}\bigg]\rd\my\\
\displaystyle\int_D\left(\frac{\eps(\mr)-\epsb}{\epsb\mub}\right)e^{-i\kb\vv_2\cdot\my}\bigg[\sum_{n\in I(2)}e^{ik\vt_n\cdot(\mx-\my)}\bigg]\rd\my\\
\vdots\\
\medskip\displaystyle\int_D\left(\frac{\eps(\mr)-\epsb}{\epsb\mub}\right)e^{-i\kb\vv_M\cdot\my}\bigg[\sum_{n\in \mathcal{I}(m)}e^{ik\vt_n\cdot(\mx-\my)}\bigg]\rd\my\\
\end{pmatrix}^T.\]

Because the following relation holds uniformly (see \cite{KP} for example)
\begin{align}
\begin{aligned}\label{JacobiAnger}
\sum_{n=1}^{N}e^{ix\cos(\theta_n-\phi)}&\approx\int_{\theta_1}^{\theta_N}e^{ix\cos(\theta-\phi)}d\theta\\
&=(\theta_N-\theta_1)J_0(x)+4\sum_{p=1}^{\infty}\frac{i^p}{p}J_p(x)\cos\frac{p(\theta_N+\theta_1-2\phi)}{2}\sin\frac{p(\theta_N-\theta_1)}{2},
\end{aligned}
\end{align}
we can evaluate
\begin{multline*}
\sum_{n\in \mathcal{I}(m)}e^{ik\vt_n\cdot(\mx-\my)}\approx\int_{\vartheta_m+\SI{60}{\degree}}^{\vartheta_m+\SI{300}{\degree}}e^{i\kb|\mx-\my|\cos(\theta-\phi)}\rd\theta\\
=\frac{4\pi}{3}J_0(\kb|\mx-\my|)+4\sum_{p=1}^{\infty}\frac{(-i)^p}{p}J_p(\kb|\mx-\my|)\cos\big(p(\vartheta_m-\phi)\big)\sin\left(\frac{2p}{3}\pi\right):=\frac{4\pi}{3}J_0(\kb|\mx-\my|)+4\mathcal{D}(\mx,\me_m).
\end{multline*}
Thus,
\begin{align*}
\mGR(\mx)\mathbb{K}\mGE(\mx)&=-\frac{1}{48\pi RE}\begin{pmatrix}
\medskip\displaystyle\int_D\left(\frac{\eps(\mr)-\epsb}{\epsb\mub}\right)e^{-i
\kb\vv_1\cdot\my}\bigg[J_0(\kb|\mx-\my|)+\frac{3}{\pi}\mathcal{D}(\mx,\me_1)\bigg]\rd\my\\
\displaystyle\int_D\left(\frac{\eps(\mr)-\epsb}{\epsb\mub}\right)e^{-i
\kb\vv_2\cdot\my}\bigg[J_0(\kb|\mx-\my|)+\frac{3}{\pi}\mathcal{D}(\mx,\me_2)\bigg]\rd\my\\
\vdots\\
\medskip\displaystyle\int_D\left(\frac{\eps(\mr)-\epsb}{\epsb\mub}\right)e^{-i
\kb\vv_M\cdot\my}\bigg[J_0(\kb|\mx-\my|)+\frac{3}{\pi}\mathcal{D}(\mx,\me_M)\bigg]\rd\my\\
\end{pmatrix}^T\begin{pmatrix}
e^{i\kb\vv_1\cdot\mx}\\
e^{i\kb\vv_2\cdot\mx}\\
\vdots\\
e^{i\kb\vv_M\cdot\mx}
\end{pmatrix}\\
&=-\frac{1}{48\pi RE}\int_D\left(\frac{\eps(\mr)-\epsb}{\epsb\mub}\right)\left[\frac{1}{M}\sum_{m=1}^{M}e^{i\kb\vv_m\cdot(\mx-\my)}\bigg(J_0(\kb|\mx-\my|)+\frac{3}{\pi}\mathcal{D}(\mx,\me_m)\bigg)\right]\rd\my.
\end{align*}
%\[\int_0^{2\pi}\cos(p\theta)\cos(p'\theta)\rd\theta=\left\{\begin{array}{ccl}
%0&\text{if}&p\ne p'\\
%\pi&\text{if}&p=p',
%\end{array}\right.\]
%we can evaluate

Now, using \eqref{JacobiAnger} again, we have
\[\sum_{m=1}^{M}e^{i\kb\vv_m\cdot(\mx-\my)}J_0(\kb|\mx-\my|)\approx\int_{\mathbb{S}^1}e^{i\kb\vv\cdot(\mx-\my)}J_0(\kb|\mx-\my|)\rd\vv=2\pi J_0(\kb|\mx-\my|)^2\]
and
\begin{align*}
&\sum_{m=1}^{M}e^{i\kb\vv_m\cdot(\mx-\my)}\mathcal{D}(\mx,\me_m)\approx\int_{0}^{2\pi}e^{i\kb|\mx-\my|\cos(\vartheta-\phi)}\left[\sum_{p=1}^{\infty}\frac{(-i)^p}{p}J_p(\kb|\mx-\my|)\cos\big(p(\vartheta-\phi)\big)\sin\left(\frac{2p}{3}\pi\right)\right]\rd\vartheta\\
&=\sum_{p=1}^{\infty}\frac{(-i)^p}{p}J_p(\kb|\mx-\my|)\sin\left(\frac{2p}{3}\pi\right)\int_{0}^{2\pi}\cos\big(p(\vartheta-\phi)\big)\bigg[J_0(\kb\mx-\my|)+2\sum_{p'=1}^{\infty}i^{p'}J_{p'}(\kb\mx-\my|)\cos\big(p'(\vartheta-\phi)\big)\bigg]\rd\vartheta\\
&=2\sum_{p=1}^{\infty}\frac{(-i)^p}{p}J_p(\kb|\mx-\my|)\sin\left(\frac{2p}{3}\pi\right)\sum_{p'=1}^{\infty}i^{p'}J_{p'}(\kb\mx-\my|)\left(\int_{0}^{2\pi}\cos\big(p(\vartheta-\phi)\big)\cos\big(p'(\vartheta-\phi)\big)\rd\vartheta\right)\\
&=2\pi\sum_{p=1}^{\infty}\frac{1}{p}J_p(\kb|\mx-\my|)^2\sin\left(\frac{2p}{3}\pi\right).
\end{align*}
Therefore,
\[\mGR(\mx)\mathbb{K}\mGE(\mx)=-\frac{1}{24RE}\int_D\left(\frac{\eps(\mr)-\epsb}{\epsb\mub}\right)\left[J_0(\kb|\mx-\my|)^2+\frac{3}{\pi}\sum_{p=1}^{\infty}\frac{1}{p}J_p(\kb|\mx-\my|)^2\sin\left(\frac{2p}{3}\pi\right)\right]\rd\my\]
and correspondingly, \eqref{Structure} can be derived.
\end{proof}

Based on Theorem \ref{Theorem}, because $J_0(0)=1$ and $J_p(0)=0$ for nonzero $p$, the value of $\mathfrak{F}(\mx)$ will reach its maximum when $\mx\in D$, allowing the existence or outline shape of objects to be recognized. This is the theoretical reason why KM is applicable when many elements of $\mathbb{K}$ are missing. Hence, we can conclude that the following result of unique determination holds.

\begin{corollary}[Unique determination]
Under the same conditions as in Theorem \ref{Theorem}, the objects $D_s$ can be identified uniquely through the map of $\mathfrak{F}(\mx)$.
\end{corollary}

\section{Simulation results obtained using experimental data}\label{sec:4}
This section presents a set of numerical simulation results from the Fresnel experimental datasets \texttt{dielTM\underline{ }dec8f.exp} and \texttt{twodielTM\underline{ }dec8f.exp} for single and two dielectric small objects, respectively. The objects are dielectric cylinders with a circular cross-section of radius $\SI{0.015}{\meter}$ and permittivity $\eps_s=3\pm0.3$ for $s=1,2$. For a detailed description of the simulation configuration, we refer to \cite{BS}.

Figure \ref{Single} depicts maps of $\mathfrak{F}(\mx)$ with different frequencies in the presence of a single object. Although the presence of the object can be detected for any frequency, its outline shape cannot be obtained when $f=\SI{1}{\giga\hertz}$.

\begin{figure}[h]
\begin{center}
\includegraphics[width=.25\textwidth]{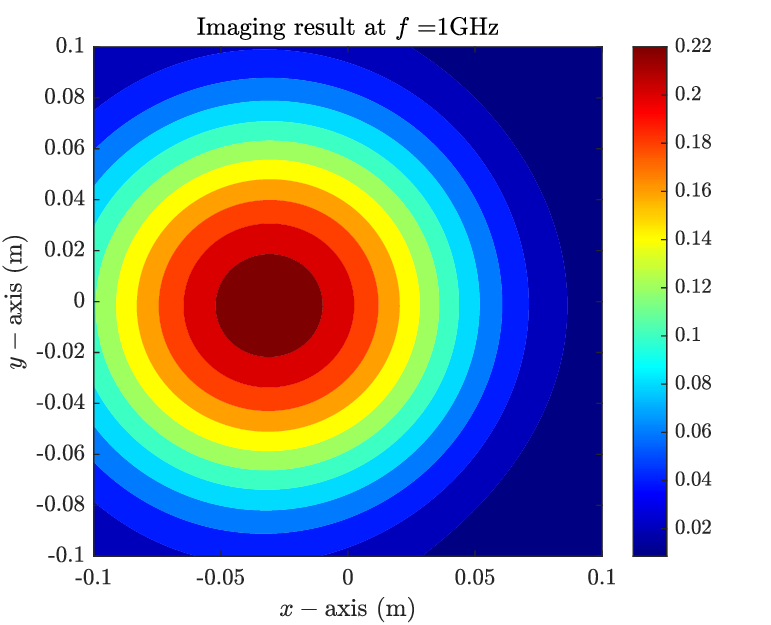}\hfill
\includegraphics[width=.25\textwidth]{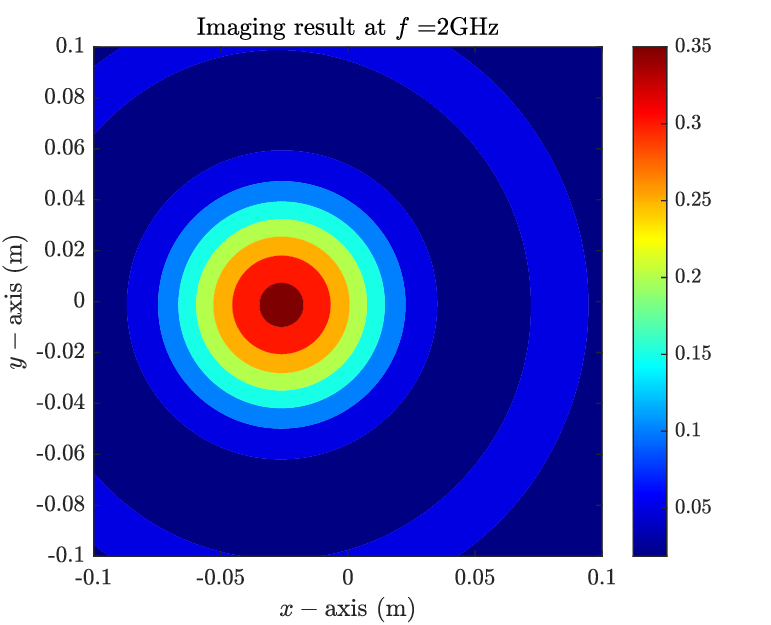}\hfill
\includegraphics[width=.25\textwidth]{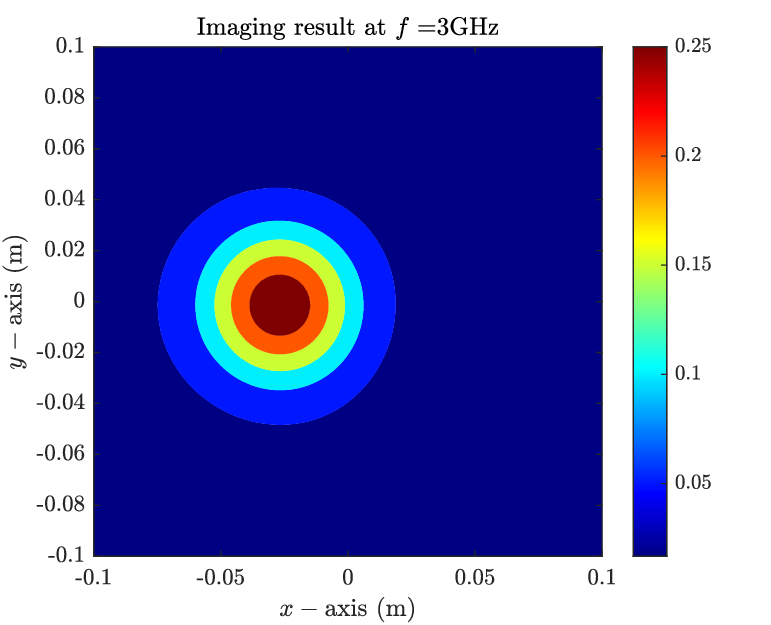}\hfill
\includegraphics[width=.25\textwidth]{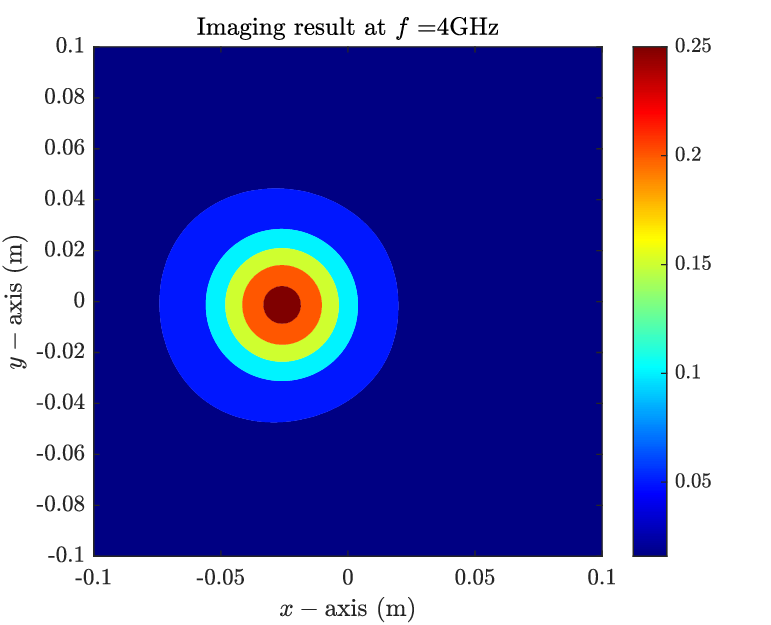}\\
\includegraphics[width=.25\textwidth]{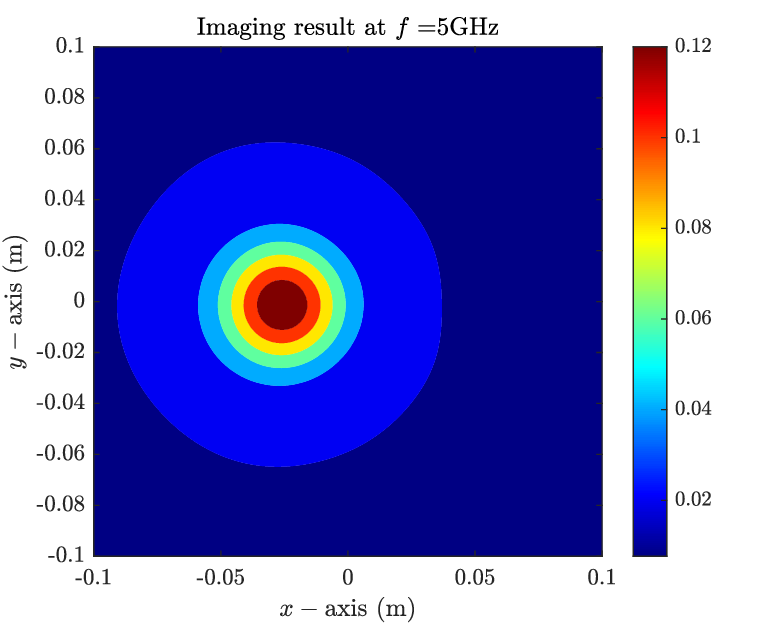}\hfill
\includegraphics[width=.25\textwidth]{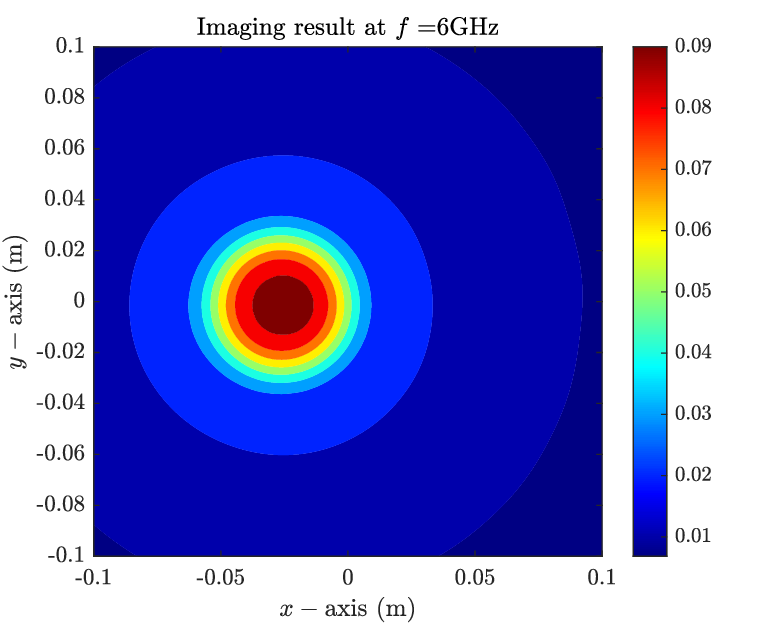}\hfill
\includegraphics[width=.25\textwidth]{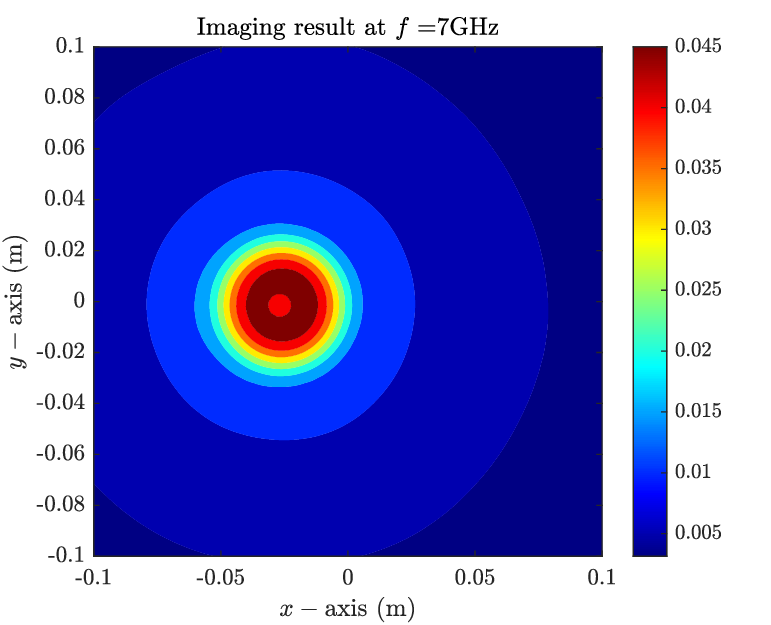}\hfill
\includegraphics[width=.25\textwidth]{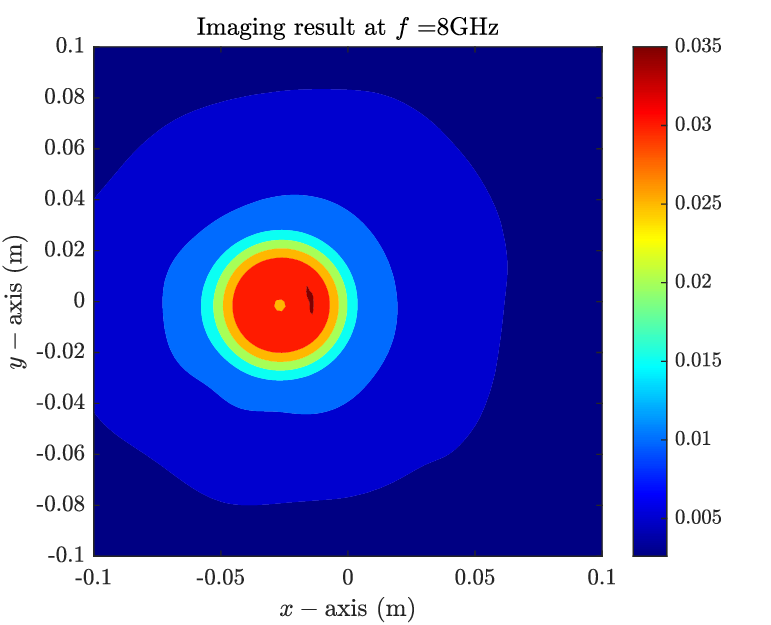}
\caption{\label{Single}Maps of $\mathfrak{F}(\mx)$ at $f=1,2,\ldots,\SI{8}{\giga\hertz}$.}
\end{center}
\end{figure}

Figure \ref{Multiple} depicts maps of $\mathfrak{F}(\mx)$ with different frequencies in the presence of two objects. Unlike identifying a single object, it is impossible to recognize two objects when $f=\SI{1}{\giga\hertz}$. It is worth noting that to distinguish two objects, half of the applied wavelength $\lambda$ should be less than the distance between them. However, because $\lambda\approx\SI{0.3}{\meter}$ when $f=\SI{1}{\giga\hertz}$ and the distance between two objects is approximately $\SI{0.09}{\meter}$, they are indistinguishable. Note that if $f\geq\SI{2}{\giga\hertz}$ then two objects can be clearly identified.

\begin{figure}[h]
\begin{center}
\includegraphics[width=.25\textwidth]{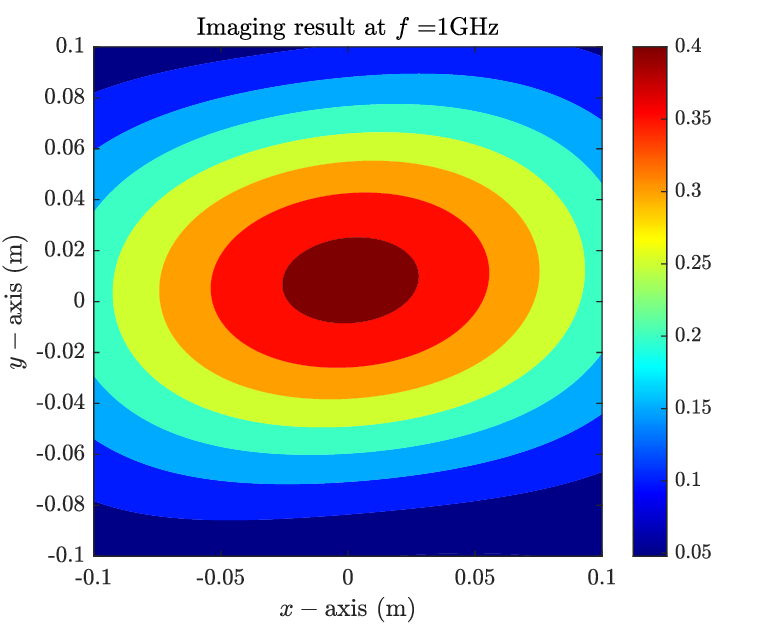}\hfill
\includegraphics[width=.25\textwidth]{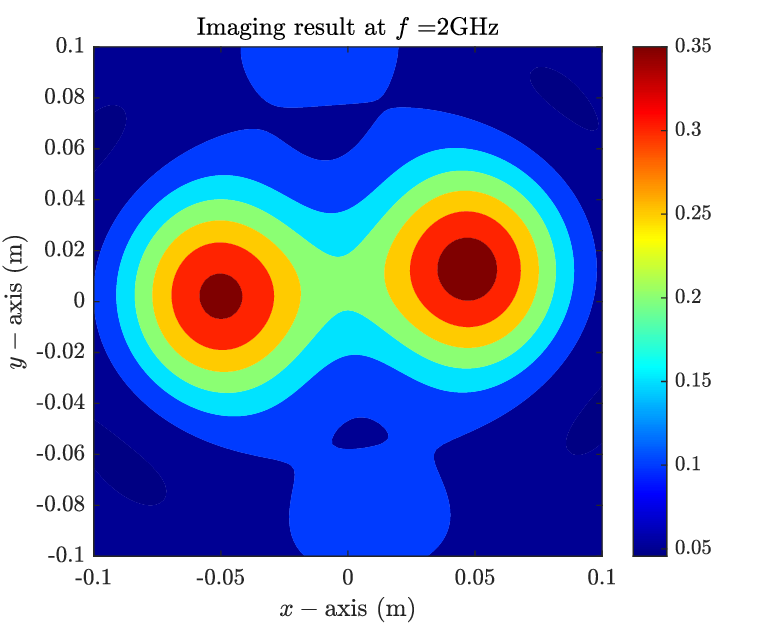}\hfill
\includegraphics[width=.25\textwidth]{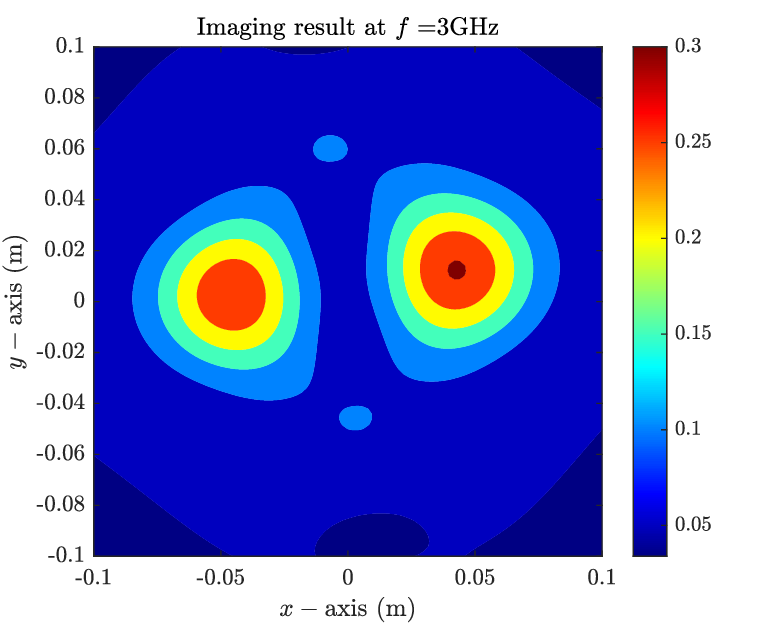}\hfill
\includegraphics[width=.25\textwidth]{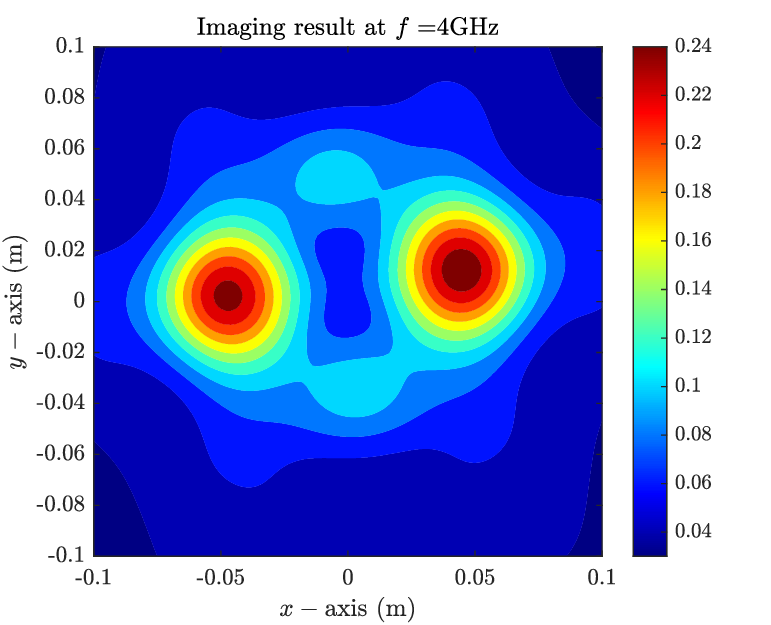}\\
\includegraphics[width=.25\textwidth]{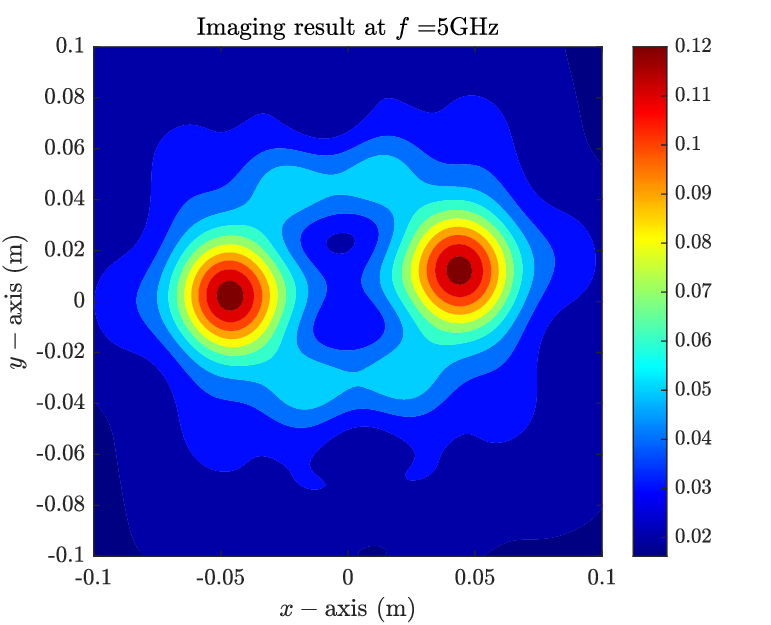}\hfill
\includegraphics[width=.25\textwidth]{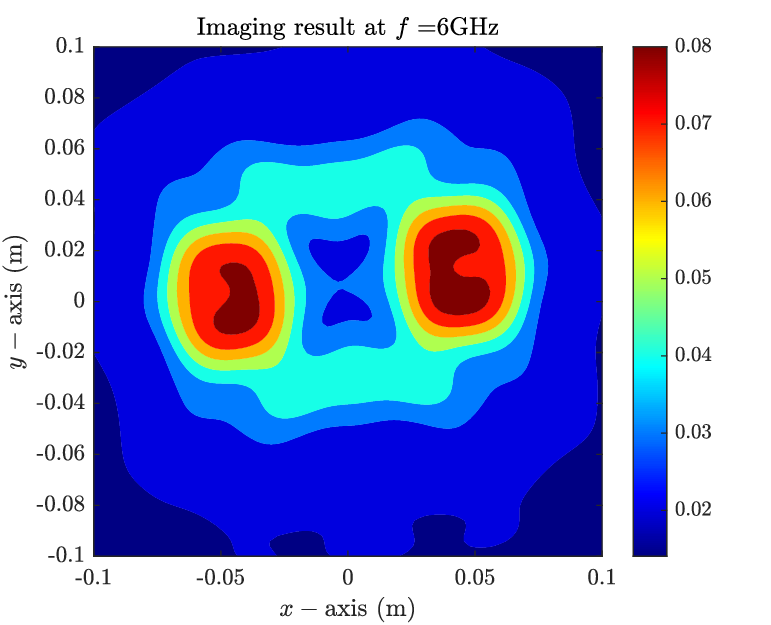}\hfill
\includegraphics[width=.25\textwidth]{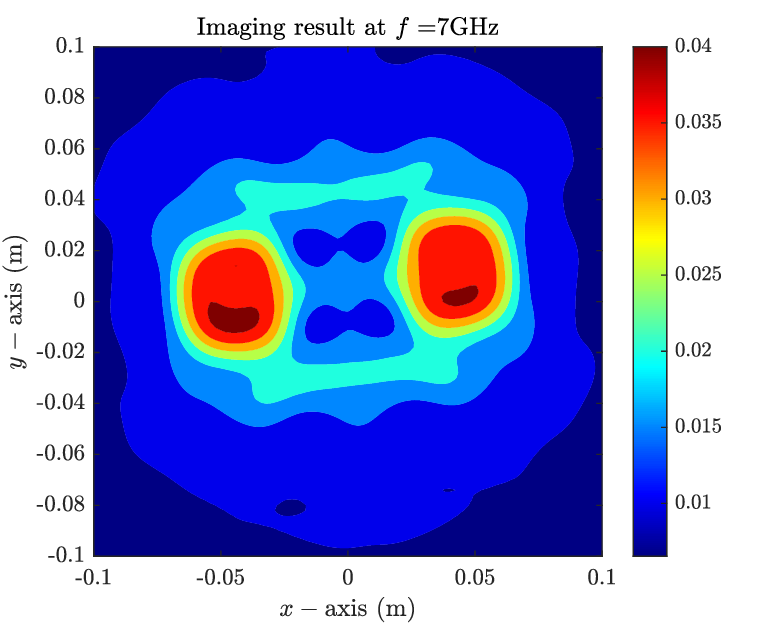}\hfill
\includegraphics[width=.25\textwidth]{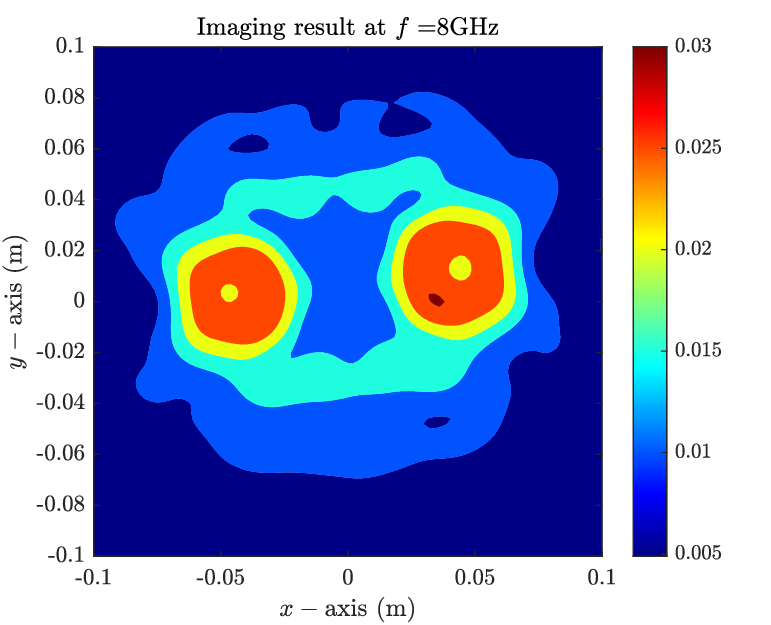}
\caption{\label{Multiple}Maps of $\mathfrak{F}(\mx)$ at $f=1,2,\ldots,\SI{8}{\giga\hertz}$.}
\end{center}
\end{figure}

\section{Conclusion}\label{sec:5}
We applied the KM for a fast imaging of small objects from experimental Fresnel dataset. By considering the relationship between the imaging function and the square of the Bessel function of order zero and the infinite series of the Bessel function of nonzero order, certain properties including the unique determination of the KM were validated. Based on the simulation results with experimental data, we concluded that the KM is an effective algorithm for detecting small objects.

The main subject of this study is the imaging of small object in two-dimensional problem. An extension to the current study using 3D Fresnel dataset \cite{GSE} will be an interesting research topic.

\section*{Acknowledgment}
This study was supported by the National Research Foundation of Korea (NRF) grant funded by the Korean government (MSIT) (NRF-2020R1A2C1A01005221).

\end{document}